\documentclass[11pt,oneside]{amsart}

\usepackage{amsmath,graphicx,amssymb}
\usepackage{hyperref}
\usepackage{graphicx}  
\title{Relative Quasiconvexity using Fine Hyperbolic Graphs}

\author[E.Mart\'inez-Pedroza]{Eduardo Mart\'inez-Pedroza}
\address{McMaster University\\ Hamilton, Ontario, Canada L8P 3E9}
\email{emartinez@math.mcmaster.ca}

\author[D.~T.~Wise]{Daniel T. Wise}
\address{McGill University \\  Montreal, Quebec, Canada H3A 2K6 }
\email{wise@math.mcgill.ca}
\newtheorem{thm}{Theorem}[section]    
\newtheorem{lem}[thm]{Lemma}          
\newtheorem{cor}[thm]{Corollary}
\newtheorem{prop}[thm]{Proposition}
\theoremstyle{definition}
\newtheorem{defn}[thm]{Definition}    
\newtheorem*{rem}{Remark}             

\DeclareMathOperator{\valence}{valence}
\DeclareMathOperator{\dist}{\mathsf{dist}}

\newcommand{\mc}{\mathcal }
\newcommand{\mb}{\mathbb }

\newcommand{\gps}{\widehat \Gamma (\mc G, \mb P, S)}
\newcommand{\N}{\ensuremath{\mathbb{N}}}

\begin{document}

\begin{abstract}    
We provide a new and elegant approach to relative quasiconvexity for relatively hyperbolic groups in the context of Bowditch's approach to relative hyperbolicity using cocompact actions on fine hyperbolic graphs. Our approach to quasiconvexity generalizes the other definitions in the literature that apply only for countable relatively hyperbolic groups. We also  provide an elementary and self-contained proof that relatively quasiconvex subgroups are relatively hyperbolic.
\end{abstract}

\maketitle


\section{Introduction}

Hruska's survey on relatively hyperbolic groups~\cite{HK08}  provides foundational work on equivalent notions of quasiconvexity for countable relatively hyperbolic groups. Almost all characterizations of relative hyperbolicity have a corresponding notion of relatively quasiconvex subgroup~\cite{HK08, MM09, Os06}. However, a definition of relatively quasiconvex subgroup within the framework of relative hyperbolicity defined in terms of a cocompact action on a hyperbolic space had not yet been pursued. In particular,  \cite{HK08} does not examine quasiconvexity in the context of Bowditch's approach to relative hyperbolicity in terms of groups acting cocompactly on fine hyperbolic graphs.

In this paper, we introduce a definition of quasiconvex subgroup in the context of relatively hyperbolic groups acting on fine hyperbolic graphs. Our notion applies for all countable and a class of uncountable relatively hyperbolic groups. We prove that our notion is equivalent to the definitions studied in~\cite{HK08} for countable relatively hyperbolic groups. We also prove that our notion of relatively quasiconvex subgroup implies relative hyperbolicity, extending  one of main results in~\cite{HK08}.  Our approach is conceptually simpler than the previous definitions in the literature, it applies to a broader class of relatively hyperbolic groups than the previous approaches, and we feel it provides a natural viewpoint.

\begin{defn}[Fine Graph]
A \emph{graph} is a  1-dimensional complex. A \emph{circuit} in a graph is an embedded cycle. A graph $\mc K$ is \emph{fine} if each $1$-cell of $\mc K$ is contained in only finitely many circuits of length~$n$ for each $n$.
\end{defn}

The following was introduced by Bowditch~\cite[Def~2]{BO99}, and we refer the reader to \cite{BO99, HK08} for its equivalence with other definitions of relative hyperbolicity for the class of countable groups. Our definition does not assume the group to be countable.
\begin{defn}[Relatively Hyperbolic Group]\label{def:BowRH}
A group $\mc G$ is \emph{hyperbolic relative to a finite collection of subgroups $\mathbb P$} if
$\mc G$ acts (without inversions) on a connected, fine, hyperbolic graph $\mc K$ with finite edge stabilizers, finitely many orbits of edges,  and $\mathbb P$ is a set of representatives of distinct conjugacy classes of vertex stabilizers (such that each infinite stabilizer is represented).

We shall refer to a connected, fine, hyperbolic graph $\mc K$ equipped with such an action as a \emph{$(\mc G, \mb P)$-graph.} Subgroups of $\mc G$ that are conjugate into subgroups in $\mb P$ are \emph{parabolic subgroups}.
\end{defn}

Our definition of relatively quasiconvex subgroup in the context of relatively hyperbolic groups acting on fine hyperbolic graphs is the following:
\begin{defn}[(${\text{Q-0}}$) Relatively Quasiconvex Subgroup]\label{def:ours}
A subgroup $\mc H$ of $\mc G$ is \emph{quasiconvex relative to  $\mb P$} if for some $(\mc G, \mb P)$-graph $\mc K$,
there is a non-empty connected and quasi-isometrically embedded subgraph $\mc L$ of $\mc K$ that is $\mc H$-invariant and has finitely many $\mc H$-orbits of edges.
\end{defn}

Our first main result states that in Definition~\ref{def:ours}, ``for some $(\mc G, \mb P)$-graph" can be replaced by ``for every $(\mc G, \mb P)$-graph", namely,
\begin{thm}\label{thm:main2}
Relative quasiconvexity of Definition~\ref{def:ours} is independent of the choice of $(\mc G, \mb P)$-graph.
\end{thm}

\begin{defn}[Finite Relative Generation]
 Let $\mc G$ be a group and  $\mb P$ a finite collection of subgroups of $\mc G$.

A set $S \subset \mc G$ is a \emph{relative generating set} for the pair $(\mc G,\mb P)$ if the set $S \cup \bigcup_{\mc P \in \mb P} \mc P$ is a generating set for $\mc G$ in the standard sense. If there is a finite relative generating set for $(\mc G, \mb P)$, we say that $\mc G$ is \emph{finitely generated relative to} $\mb P$.

A subgroup $\mc H$ of $\mc G$ is \emph{finitely generated relative to $\mb P$} if
there is a finite collection of subgroups $\mb R$ of $\mc H$ such that $\mc H$ is finitely generated relative to $\mb R$ and each subgroup $\mc R \in \mb R$ is conjugate  in $\mc G$ into a subgroup $\mc P \in \mb P$.
\end{defn}

The following is Hruska's version in~\cite{HK08} of Osin's definition of relative hyperbolicity in~\cite{Os06} for countable groups.

\begin{defn}[($\overline{\text{Q-1}}$) Osin Quasiconvex Subgroup in $\overline \Gamma$] \label{def:O}
Suppose that $\mc G$ is hyperbolic relative to $\mb P$  and $S$ is a finite relative generating set for $(\mc G, \mb P)$.
Let $\overline \Gamma = \overline \Gamma (\mc G, S\cup \mb P)$ be the Cayley graph of $\mc G$ with respect to the generating set $S \cup \bigcup_{\mc P \in \mb P} \mc P$, and let $\dist$ be a proper left-invariant metric on $\mc G$.

A subgroup $\mc H$ of $\mc G$ is \emph{quasiconvex relative to  $\mb P$} if there exists a constant $\sigma \geq 0$ such that the following holds: Let $f, g$ be two elements of $\mc H$, and let $P$ be an arbitrary geodesic path from $f$ to $g$ in $\overline \Gamma$.  For any vertex $p$ in $P$, there exists a vertex $h$ in $\mc H$ such that $\dist (p,h) \leq \sigma$.
\end{defn}

Every countable group is a subgroup of a finitely generated group, and therefore a group is countable if and only if it admits a proper left invariant metric. The second main result of the paper is the following equivalence:
\begin{thm}\label{thm:main}
For countable relatively hyperbolic groups, relative quasiconvexity of Definition~\ref{def:ours} is equivalent to relative quasiconvexity of Definition~\ref{def:O}.
\end{thm}

In~\cite{Os06}, Osin asked whether relative quasiconvexity of Definition~\ref{def:O} implies relative hyperbolicity with respect to the maximal parabolic subgroups.
This question was positively answered by Hruska in~\cite{HK08} using the convergence group approach to relative hyperbolicity and results of Tukia in~\cite{Tu98}.
Evidence of the naturality of  Definition~\ref{def:ours} is that it permits a short and self-contained alternative proof of a more general version of Hruska's result.

\begin{thm}
Let $\mc G$ be hyperbolic relative to $\mb P$. If $ \mc H < \mc G$ is relatively quasiconvex, then $\mc H$ is relatively hyperbolic with respect to a collection of parabolic subgroups of $\mc H$.
\end{thm}
\begin{proof}
Let $\mc K$ be a $(\mc G, \mb P)$-graph. Since $\mc H$ is relatively quasiconvex, there is a nontrivial connected and quasi-isometrically embedded subgraph $\mc L \subset \mc K$ which is $\mc H$-invariant and has finitely many $\mc H$-orbits of edges.  Since a subgraph of a fine graph is fine, and a quasi-isometrically embedded subspace of a hyperbolic space is hyperbolic, the graph $\mc L$ is hyperbolic and fine. Since $\mc G$-stabilizers of edges of $\mc K$ are finite, $\mc H$-stabilizers of edges of $\mc L$ are finite. Therefore $\mc H$ is hyperbolic relative to a finite collection of stabilizers of vertices of $\mc L$.
\end{proof}

Another corroboration of the naturality of Definition~\ref{def:ours} is that it allows us to correctly interpret results for countable groups acting on small cancellation complexes in the context of relative hyperbolicity and coherence of groups~\cite{MaWi10}.

\subsection*{Outline:}
The paper consists of three sections. The first section contains the proof of Theorem~\ref{thm:main2}. The second section shows a relation between fellow traveling of quasigeodesics in hyperbolic spaces and the existence of narrow disc diagrams between quasigeodesics. This  relation combined with the notion of fine graphs allows us to deduce a strong fellow travel property for fine hyperbolic graphs admitting cocompact actions.  The last section contains the proof of Theorem~\ref{thm:main}.

\subsection*{Acknowledgments:}
We thank the referee for critical corrections, and Inna Bumagin for useful comments. The first author acknowledges the support of the Geometry and Topology group at McMaster University through a Postdoctoral Fellowship, and partial support of the Centre de Recherches Math\' ematiques in Montreal to attend some of the Fall-2010 events during which part of this paper was prepared. The second author's research is supported by NSERC.

\section{Independence of Quasiconvexity} \label{sec:fine-graphs}

In this section, we prove that Definition~\ref{def:ours} is independent of the $(\mc G, \mb P)$-graph. This is restated in this section as Theorem~\ref{thm:movings}.
The proof is based on Theorem~\ref{thm:main4} which is a result on equivariant embeddings between $(\mc G, \mb P)$-graphs.

\subsection{Preliminary results}

The results on fine graphs  discussed below essentially all appeared  in the work of Bowditch~\cite{BO99}. We provide proofs for the convenience of the reader.

\begin{lem}\label{lem:finite-paths}
Let $\mc K$ be a graph. The following statements are equivalent.
\begin{enumerate}
\item $\mc K$ is fine.
\item  For each integer $n>0$, and any pair of vertices $u, v$ of $\mc K$, there are only finitely many embedded paths of length $n$ between $u$ and $v$.
\end{enumerate}
 \end{lem}
\begin{proof}
Suppose that $\mc K$ is fine, $n>0$, and $u,v \in \mc K$. Suppose that  $\{ P_i : i \in \N  \}$ is a collection of distinct embedded length~$n$ paths between $u$ and $v$.  For each $i>1$, the (closure of the) symmetric difference of  $P_1 \triangle P_i$ consists of a collection of embedded cycles each of which has an edge in $P_1$.
As all these cycles have length $\leq n$, we arrive a contradiction with the fineness of $\mc K$.

For the other direction, notice that length~$n$ circuits containing an edge $e$ with endpoints $u,v$, are in bijective correspondence with the embedded paths of length $n-1$ between $u,v$ that do not contain the edge $e$.
\end{proof}

\begin{lem}[Almost Malnormal]\label{lem:malnormality}
Let $\mc G$ act on a fine graph $\mc K$ with finite edge stabilizers.
For vertices $u,v$ the intersection $\mc G_u \cap \mc G_v$ is finite unless $u=v$.
\end{lem}
\begin{proof}
Suppose that $u\neq v$ and $\mc M = \mc G_u \cap \mc G_v$ is an infinite subgroup.
Let $P$ be an embedded path from $u$ to $v$. By Lemma~\ref{lem:finite-paths},  there are only finitely many $\mc M$-translates of $P$. Since $\mc M$ is assumed to be infinite, the path $P$ has an infinite $\mc G$-stabilizer. In particular, there is an edge with infinite $\mc G$-stabilizer, and this contradicts that $\mc K$ has finite $\mc G$-stabilizers of edges.
\end{proof}

\begin{lem}[Infinite valence $\Leftrightarrow $ Infinite stabilizer]\label{lem:infinite-stabilizers}
Let $\mc G$  act cocompactly on a graph $\mc K$ with finite edge stabilizers.
Then a vertex $v \in \mc K$ has infinite valence if and only if  its stabilizer $\mc G_v$ is infinite.
\end{lem}
\begin{proof}
Since there are only finitely many $\mc G$-orbits of edges, if $v$  has infinite valence, then $v$ infinite stabilizer.
Conversely, since $\mc G$-stabilizers of edges  are finite, if $v$ has infinite stabilizer, then $v$ has infinite valence.
\end{proof}

\begin{lem}[Infinite Valence Vertices are Canonical]\label{prop:bijection}
Let $\mc G$ be hyperbolic relative to a collection of subgroups $\mb P$, and let $\mc K$ be a $(\mc G, \mb P)$-graph.
Let $\mb P_\infty$ be the subcollection of $\mb P$ consisting of infinite subgroups, and let $V_\infty (\mc K)$ be the set of infinite valence vertices of $\mc K$.
There is a natural $\mc G$-equivariant bijection
\[ V_\infty (\mc K)   \longrightarrow   \left \{  g\mc Pg^{-1}\ : \ g\in \mc G,\  \mc P \in \mb P_\infty  \right \} \]
that maps a vertex $v$ to its $\mc G$-stabilizer $\mc G_v$.
\end{lem}
\begin{proof}
\emph{Range of the map is well-defined.}  By Lemma~\ref{lem:infinite-stabilizers}, if a vertex $v$ has infinite valence, then $v$ has infinite stabilizer. By definition of $(\mc G, \mb P)$-graph, if $v$  has infinite stabilizer, then  $\mc G_v = g\mc Pg^{-1}$ for some $g\in \mc G$ and $\mc P \in \mb P_\infty$.

\emph{Surjectivity.} Every subgroup of the form $g\mc Pg^{-1}$ for  $g\in \mc G$ and $\mc P \in \mb P_\infty$ is the $\mc G$-stabilizer of a vertex $v$ of $\mc K$.
In this case $v$ has infinite $\mc G$-stabilizer and hence  Lemma~\ref{lem:infinite-stabilizers} implies that  $v$ has infinite valence.

\emph{Injectivity.} Follows from Lemma~\ref{lem:malnormality}.
\end{proof}

The following Corollary of Lemma~\ref{prop:bijection} is easily obtained directly.
\begin{cor}\label{cor:infinite-valence-2}
Let $\mc G$ be a hyperbolic group relative to a collection of subgroups $\mb P$, and let $\mc K_2 \hookrightarrow  \mc K_1$ be a $\mc G$-equivariant embedding of  $(\mc G, \mb P)$-graphs. Then every infinite valence vertex of $\mc K_1$ is in $\mc K_2$.
\end{cor}

\begin{defn}[Equivariant Arc Attachment]
Let $\mc J$ be a graph admitting an action of a group $\mc G$, let $\mc K$ be a subgraph of a graph $\mc J$, and let $P$ be a path in $\mc J$.
The \emph{$\mc G$-attachment of the arc $P$ to $\mc K$} means forming the  new subgraph
\[\mc K' = \mc K \cup  \bigcup_{g\in \mc G} gP.\]
\end{defn}

\begin{lem}[Arc Attachment Preserves Coarse Geometry]  \label{prop:arc-attachment-2}
Let $\mc G$ act on a graph $\mc J$, and let $\mc K$ be a connected $\mc G$-invariant subgraph of $\mc J$.
Suppose $\mc K'$ is obtained from $\mc K$ by a $\mc G$-attachment of an arc $P$ with at least one of its endpoints in $\mc K$.
Then the inclusion $\mc K \subset \mc K'$ is a quasi-isometry. In particular, if $\mc K$ is hyperbolic, then $\mc K'$ is hyperbolic.
\end{lem}
\begin{proof}
Without loss of generality we can assume that no interior points of $P$ belong to $\mc K$. If the interior of $P$ intersects $\mc K$,
then the $\mc G$-attachment of $P$ is equivalent to a finite number of $\mc G$-attachments of paths with no interior points in $\mc K$.

If $P$ has only one endpoint in $\mc K$, then $\mc K \subset \mc K'$ is an isometric embedding. Assume that both endpoints of $P$ are in $\mc K$, and
let $P_0$ be a geodesic path in the connected graph $\mc K$ connecting the endpoints of $P$.

A geodesic path $Q'$ in  $\mc K'$  yields a path $Q \subset \mc K$  by replacing all $\mc G$-translates $gP$ occurring in $Q'$ by the path $gP_0$.
Observe that $|Q| \leq |Q'||P|$. Therefore, if $\dist_{\mc K}$ and $\dist_{\mc K'}$ denote the path metrics of $\mc K$ and $\mc K'$ respectively, then $ \dist_{\mc K}(u, v) \leq  |P| \dist_{\mc K'}(u, v)$ for any $u,v \in \mc K$.
\end{proof}

\begin{lem}[Single Edge Attachment] \label{lem:edge-attachment}
Let $\mc G$ act on a graph $\mc J$ with finite stabilizers of edges,  let $\mc K$ be a connected $\mc G$-invariant fine subgraph of $\mc J$,
and let $\mc K'$ be a subgraph of $\mc J$ obtained from $\mc K$ by the $\mc G$-attachment of an edge $e$ between two vertices of $\mc K$.
Then for each $n\in \N$ and each pair of vertices $u,v$ of $\mc K'$, there is a finite subgraph $\mc C=\mc C(u,v,n)$ of $\mc K$
such that any length~$n$ embedded path in $\mc K'$ between $u,v$  has all vertices contained in $\mc C$.
\end{lem}
\begin{proof}
Let $P_0$ be a path in $\mc K$ between the endpoints of $e$. Consider the following two operations on a subgraph $\mc C$ of $\mc K$.
\begin{enumerate}
\item  ($n|P_0|$-hull in $\mc K$)
Add all embedded paths in $\mc K$ of length $\leq n|P_0|$ with different endpoints in $\mc C$.
\item ($P_0$-inclusion)
Add each translated $gP_0$ (for $g\in \mc G$) containing at least one edge of $\mc C$.
\end{enumerate}
Note that the above operations preserve finiteness. Since $\mc K$ is fine, Lemma~\ref{lem:finite-paths} implies that $n|P_0|$-hulls preserve  finiteness.
A $P_0$-inclusion  preserves finiteness since $\mc G$-stabilizers of edges are finite. 

Let $u,v$ be different vertices of $\mc K'$ and $n\in \N$.  Let $\mc C_0 = \{u, v\}$, and let $\mc C_{i+1}$ be the finite graph obtained from $\mc C_i$ by performing an $n|P_0|$-hull and then a $P_0$-inclusion. Let $\mc C = \mc C^n$.  

Let $Q'$ be an embedded path in $\mc K'$ from $u$ to $v$ of length $\leq n$.  Suppose that $\mc C_i$ does not contain all vertices of $Q'$. We will then show that $\mc C_{i+1}$ contains more vertices of $Q'$ than $\mc C_i$. 

If some new edge $ge$ of $Q'$ whose endpoints are not both in $\mc C_i$  has the property that $gP_0$ has a common edge with $\mc C_i$, then the endpoints of $ge$ are in $\mc C_{i+1}$. Hence $\mc C_{i+1}$ contains more vertices of $Q'$ than $\mc C_i$.

Assume that no new edge $ge$ of $Q'$ has the above property.
Consider a maximal subpath $S'$ of $Q'$ whose internal vertices do not lie in $\mc C_i$ and contains at least one vertex of $Q'$ that is not in $\mc C_i$.
Let $S$ denote the subgraph of $\mc K$ which is obtained from $S'$ by replacing each $ge$ by $gP_0$.
By the assumption, $S\cap \mc C_i$ has no edge. Moreover  $S$ is connected, contains the endpoints of $S'$, and has at most $n|P_0|$ edges.
It follows that there is an embedded path $E$ in $S$ of length $\leq n|P_0|$ joining the endpoints of $S'$.
Observe that all interior vertices of $E$ are outside of $\mc C_i$, and $E$ is contained in the $n|P_0|$-hull of $\mc C_i$.

Now we consider two cases on $E$. Either $E$ contains an edge of $Q'$ with at least one endpoint not in $\mc C_i$, or $E$ contains an edge of $gP_0$ where $ge$ is an edge of $Q'$ with at least one endpoint not in $\mc C_i$. In both cases, this new endpoint is a vertex of $Q'$ that is in $C_{i+1}-C_i$. Hence $\mc C_{i+1}$ contains more vertices of $Q'$ than $\mc C_i$.
\end{proof}

A proof of Lemma~\ref{prop:arc-attachment} can also be found in~\cite[Lem~2.3]{BO99}.
\begin{lem} [Arc Attachment Preserves Fineness]  \label{prop:arc-attachment}
Let $\mc G$ act on a graph $\mc J$ with finite stabilizers of edges.
Let $\mc K$ be a connected $\mc G$-invariant subgraph of $\mc J$,
and let $\mc K'$ be obtained  from $\mc K$ by the $\mc G$-attachment of an arc $P$.
Then if $\mc K$ is fine, then $\mc K'$ is fine.
\end{lem}
\begin{proof}
Without loss of generality we may assume that no interior points of $P$ belong to $\mc K$. Indeed, if the interior of $P$ intersects $\mc K$,
then the $\mc G$-attachment of $P$ is equivalent to a finite number of $\mc G$-attachments of paths with no interior points in $\mc K$.
Observe that if $P$ has only one endpoint in $\mc K$, then every circuit in $\mc K'$ is contained in $\mc K$,  and therefore $\mc K'$ is fine.
It therefore suffices to consider the case that $P$ consists of a single edge between a pair of vertices of $\mc K$.

Observe that $\mc K'$ has finitely many edges between any pair of vertices. Indeed, since $\mc K$ is fine, it has finitely many edges between any pair of vertices;
then, since $\mc G$ acts with finite edge stabilizers on $\mc J$,  Lemma~\ref{lem:malnormality} implies the statement. 

Let $u,v$ be distinct vertices of $\mc K'$ and fix $n>0$. By Lemma~\ref{lem:edge-attachment}, there is a finite subgraph $\mc C$ of $\mc K$
such that any length~$n$ embedded path in $\mc K'$ between $u,v$ has all vertices contained in $\mc C$. 
Since $\mc K'$ has finitely many edges between any pair of vertices, there are only finitely many length~$n$ embedded paths between $u,v$ in $\mc K'$.
By Lemma~\ref{lem:finite-paths}, $\mc K'$ is fine.
\end{proof}

\begin{defn} [Edge and Vertex Removals]
Let $\mc G$ be a group acting on a graph $\mc K$.

If $e$ is an edge of $\mc K$, the \emph{$\mc G$-removal of the edge $e$} means forming the new graph
$\mc K'$ obtained by removing the \emph{interiors} of all $\mc G$-translates of $e$.

If $v$ is an edge of $\mc K$, the \emph{$\mc G$-removal of the vertex $v$} means forming the  the new graph
$\mc K'$ obtained by removing  all $\mc G$-translates of $v$ and  all $\mc G$-translates of open edges with an endpoint at $v$.
\end{defn}

\begin{lem}[Removals preserve fineness and coarse geometry]\label{prop:removals}
Let $\mc G$ act cocompactly on a connected graph $\mc K$ with finite $\mc G$-stabilizers of edges.
Let $\mc K'$ be the graph obtained from $\mc K$ by performing a $\mc G$-removal of a finite valence vertex,
or a $\mc G$-removal of an edge.
\begin{itemize}
\item If $\mc K'$ is connected, then the inclusion $\mc K' \subset \mc K$ is a quasi-isometry. In particular, if $\mc K$ is hyperbolic  then $\mc K'$ is hyperbolic.
\item If $\mc K$ is fine  then $\mc K'$ is fine.
\end{itemize}
\end{lem}
\begin{proof}
That fineness is preserved under edge $\mc G$-removals and finite valence vertex $\mc G$-removals is immediate. We address the quasi-isometric embedded  property.

\emph{Edge $\mc G$-removal.} Suppose that $P$ is an edge of $\mc K$ and that $\mc K'$ is connected. Let $\dist_{\mc K}$ and $\dist_{\mc K'}$ denote the combinatorial path metrics of $\mc K$ and $\mc K'$ respectively. Let $Q$ be a path in $\mc K'$ with the same endpoints as $P$, and let $M=|Q|$. A standard argument shows that $\dist_{\mc K'} (u, v) \leq M \dist_{\mc K} (u, v)$ for any pair of vertices $u, v$ of $\mc K$.  Hence,  the inclusion $\mc K' \subset \mc K$ is a quasi-isometric embedding.

\emph{Finite valence vertex $\mc G$-removal.} Observe that when $\valence (v) \geq 2$, then an edge at $v$ can be $\mc G$-removed.  Repeating this finitely many times, we arrive at the situation where $\valence (v)=1$. We now remove all $\mc H$-translates of the \emph{spur} consisting of the vertex $v$ together with its unique adjacent edge. Since edge and spur $\mc G$-removals induce quasi-isometric embeddings, the inclusion $\mc K' \subset \mc K$ is a quasi-isometric embedding.
\end{proof}

\subsection{Joint Equivariant Embedding of Two Fine Graphs}

\begin{thm}\label{thm:main4}
Let $\mc G$ be hyperbolic relative to a collection of subgroups $\mb P$, and let $\mc K_1$ and $\mc K_2$ be $(\mc G, \mb P)$-graphs.
Then there is a $(\mc G, \mb P)$-graph $\mc K$ such that $\mc K_1$ and $\mc K_2$ both embed equivariantly and simplicially into $\mc K$.
\end{thm}
\begin{proof}
For each $\mc P \in \mb P$, choose    vertices $u_{ _{\mc P}} \in \mc K_1$ and $v_{ _{\mc P}} \in \mc K_2$ having $\mc G$-stabilizer $\mc P$. Observe that by Lemma~\ref{prop:bijection}, if $\mc P$ is infinite there are unique choices for $u_{ _{\mc P}}$ and $v_{ _{\mc P}}$.

Let \[ V_{ _{\mb P}}(\mc K_1) = \{ g u_{ _{\mc P}} \ : \ g \in \mc G, \  \mc P\in \mb P \}, \]
and
\[ V_{ _{\mb P}}(\mc K_2) = \{ g v_{ _{\mc P}} \ : \ g \in \mc G, \  \mc P\in \mb P \}. \]
There is a natural $\mc G$-equivariant bijection $\varphi : V_{ _{\mb P}}(\mc K_1) \longrightarrow V_{ _{\mb P}}(\mc K_2)$ given by $g u_{ _{\mc P}} \mapsto g v_{ _{\mc P}}$ for each $g \in \mc G$ and $\mc P \in \mb P$.

Let $\mc K$ be the graph obtained from the disjoint union of $\mc K_1$ and $\mc K_2$ by identifying $V_{ _{\mb P}} (\mc K_1)$ with  $V_{ _{\mb P}} (\mc K_2)$ via the $\mc G$-equivariant map $\varphi$.  By construction, $\mc G$  acts on $\mc K$ with finitely many $\mc G$-orbits of edges, and with finite $\mc G$-stabilizers of edges. Moreover, $\mc K_1$ and $\mc K_2$ have natural $\mc G$-equivariant inclusions into $\mc K$.

By Corollary~\ref{cor:infinite-valence-2}, each vertex of $\mc K - \mc K_1$ has finite valence. Since $\mc K$ contains only finitely many $\mc G$-orbits of edges,  one obtains $\mc K$ after finitely many $\mc G$-equivariant arc attachments to $\mc K_1$. Since $\mc K_1$ is hyperbolic and fine,  Lemmas~\ref{prop:arc-attachment-2} and~\ref{prop:arc-attachment} imply that the graph $\mc K$ is hyperbolic and fine, and the inclusion $\mc K_1 \subset \mc K$ is a quasi-isometric embedding.
\end{proof}

\subsection{Independence of $(\mc G, \mb P)$-graph}

\begin{lem}\label{lem:initial-complex}
Let $\mc H$ be a group acting on a connected graph $\mc K$.  If $\mc H$  is finitely generated relative a finite collection of stabilizers of vertices of $\mc K$, then $\mc H$ acts cocompactly on a connected subgraph $\mc L$ of $\mc K$.

Specifically, suppose that $\mc H$ is generated by a finite subset $T \subset \mc H$ relative to the $\mc H$-stabilizers of  the vertices $C$.
If $v$ is a vertex of $\mc K$  and  $\mc J$ is a finite connected subgraph of $\mc K$ containing  the vertices $\{v \}  \cup Tv \cup C$,
then the  graph $\mc L = \bigcup_{h\in \mc H} h\mc J$ is connected.
\end{lem}
\begin{proof}
Since $\mc K$ is connected, there is a finite connected subgraph $\mc J$ of $\mc K$ containing  $\{v \}  \cup Tv \cup C$.
For any $h \in T \cup \bigcup_{i} H_{v_i}$, observe that $\mc J \cap h \mc J \neq \emptyset$.
Indeed, if $h\in T$ then $hv \in \mc J \cap h \mc J$, and if $h \in \mc H_{v_i}$ then $v_i \in \mc J \cap h \mc J$.
Therefore $\mc L=\bigcup_{h\in \mc H} h\mc J$ is connected.
Moreover, since $\mc J$ is compact, $\mc H$ acts cocompactly on $\mc L$.
\end{proof}

We restate and prove Theorem~\ref{thm:main2} below:
\begin{thm}[Quasiconvexity Independence of $\mc K$]\label{thm:movings}
Let $\mc K_1$ and $\mc K_2$ be $(\mc G, \mb P)$-graphs, and let $\mc H$ be a subgroup of $\mc G$.  If $\mc H$ satisfies relative quasiconvexity of Definition~\ref{def:ours} for $\mc K_1$, then it does for $\mc K_2$.   
\end{thm}
\begin{proof}
Let $\mc L_1$ be a non-empty connected and quasi-isometrically embedded subgraph of $\mc K_1$ that is $\mc H$-invariant and has finitely many $\mc H$-orbits of edges. We will construct a subgraph $\mc L_2$ of $\mc K_2$ with the same properties as $\mc L_1$.

\emph{Reducing to the case $\mc K_2 \subset \mc K_1$.}  By Theorem~\ref{thm:main4},  $\mc K_1$ and $\mc K_2$  have  $\mc G$-equivariant and quasi-isometric embeddings in a common $(\mc G, \mb P)$-graph $\mc K$.
It follows that $\mc L_1$ quasi-isometrically embeds in $\mc K$, so $\mc H$ satisfies relative quasiconvexity of Definition~\ref{def:ours} with respect
to $\mc K$. We can thus assume without loss of generality that $\mc K_2  \subset \mc K_1$.

\emph{Vertices of $\mc K_1$  with infinite stabilizers are contained in $\mc K_2$.}
Let $v \in \mc K_1$ be a vertex with infinite $\mc G$-stabilizer. Since $\mc K_1$ is a $(\mc G, \mb P)$-graph, $\mc G_{v} =  g  \mc P g^{-1}$ for some $\mc P \in \mb P$ and $g \in \mc G$. Since $\mc K_2$ is a $(\mc G, \mb P)$-graph, there is a vertex $w \in \mc K_2$ such that $\mc G_{w} =  g  \mc P g^{-1}$.  Since $v,w \in \mc K_1$ and have the same infinite stabilizer, Lemma~\ref{lem:malnormality} implies that $v=w$, and therefore $v \in \mc K_2$.

\emph{Producing an $\mc H$-cocompact connected subgraph $\mc L_2$ of $\mc K_2$.}
Since $\mc H$ acts cocompactly on the connected graph $\mc L_1$, there is a finite subset $T \subset \mc H$ such that $\mc H$ is generated by $T$
relative to the stabilizers of the vertices in $C=\{v_1, \dots , v_m\} \subset\mc L_1$.  By possibly enlarging $T$, we can assume that each $v_i \in C$ has infinite $\mc H$-stabilizer, and thus each $v_i$ is also a vertex of  $\mc K_2$. Let $v$ be a vertex of $\mc L_2$ and let $\mc J$ be a finite connected subgraph of $\mc K_2$ containing
the vertices $\{v\}\cup Tv \cup C$.  Let
\[  \mc L_2 = \bigcup_{h \in \mc H} h\mc J   \ \ \subset \ \ \mc K_2 ,\]
and notice that $\mc L_2$ is $\mc H$-cocompact by construction. Moreover, $\mc L_2$ is connected by Lemma~\ref{lem:initial-complex}.

\emph{Enlarging $\mc L_2$ within $\mc K_2$ to contain all infinite valence vertices of $\mc L_1$.}
Each infinite valence vertex of $\mc L_1$ lies in $\mc K_2$ by Corollary~\ref{cor:infinite-valence-2}. Since there are finitely many $\mc H$-orbits of such vertices, and since $\mc K_2$ is connected,  we can choose finitely many $\mc H$-enlargements  of $\mc L_2$ within $\mc K_2$, in order to guarantee that all such infinite valence vertices of $\mc L_1$ also lie in $\mc L_2$.

\emph{Reducing to the case $\mc L_2 \subset \mc L_1$.}
Since $\mc L_2$ has finitely many $\mc H$-orbits of edges, by Lemma~\ref{prop:arc-attachment-2},  after $\mc H$-attaching finitely many edges to $\mc L_1$, we can assume that $\mc L_2 \subset \mc L_1$.

\emph{Passing from $\mc L_1$ to $\mc L_2$ with finitely many $\mc H$-removals.}
Since each vertex in  $\mc L_1-\mc L_2$ has finite valence, and $\mc L_1$ has finitely many $\mc H$-orbits of edges,  $\mc L_2$ can be obtained from $\mc L_1$ by performing finitely many $\mc H$-removals of finite valence vertices  together with their incident edges.

Since $\mc L_1 \subset \mc K_1$ is a quasi-isometric embedding, Lemma~\ref{prop:removals} implies that
each $\mc L_2 \subset \mc K_1$ is a quasi-isometric embedding.  Since this inclusion factors as $\mc L_2 \subset \mc K_2 \subset \mc K_1$,
we see that $\mc L_2 \subset \mc K_2$ is also a quasi-isometric embedding.

We have thus reached our conclusion, since it is already true (by construction) that  $\mc L_2$ is a non-empty connected $\mc H$-invariant subgraph of $\mc K_2$ having finitely many $\mc H$-orbits of edges. In particular $\mc H$ satisfies relative quasiconvexity of Definition~\ref{def:ours} for $\mc K_2$.
\end{proof}

\section{Simple Ladders between Quasigeodesics}\label{sec:ladder}

In this section, we show that the fellow traveling of quasigeodesics in $\delta$-hyperbolic spaces is equivalent to the existence of narrow disc diagrams between  quasigeodesics. This is stated as Proposition~\ref{prop:upper}. As an application, we prove a strong fellow travel property for fine hyperbolic graphs admitting cocompact actions, Theorem~\ref{thm:BCP}.

\begin{defn}[$X_n(\mc K)$]
Recall that a \emph{circuit} is a (combinatorial) embedded circle.
If $\mc K$ is a graph and $n$ is a positive integer, the $2$-complex $X_n(\mc K)$ is constructed by attaching a $2$-cell along each circuit of length at most $n$.
\end{defn}

\begin{defn}[Simple Ladder]
A \emph{simple ladder between $P$ and $Q$} is a nonsingular disc diagram $D$
that is the union of a sequence of $2$-cells $R_1,\dots, R_\ell$ such that
each $R_i$ intersects $P$ and $Q$ in a nontrivial boundary arc,
and $R_i\cap R_j$ is a nontrivial internal arc when $|i-j|=1$,
and $R_i\cap R_j =\emptyset$ when $|i-j|>1$,
finally the startpoint of $P,Q$ lies in the interior of $R_1\cap \partial D$,
and the endpoint of $P,Q$ lies in the interior of $R_\ell \cap \partial D$. See Figure~\ref{fig:SimpleLadder}.
\end{defn}

\begin{figure}\centering
\includegraphics[width=.4\textwidth]{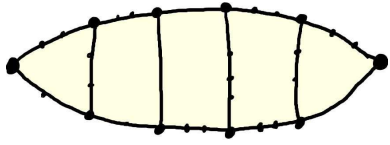}
\caption{A simple ladder between two paths. \label{fig:SimpleLadder}}
\end{figure}

\begin{defn}[Quasigeodesic]
Let $\mc K$ be a graph and let $\dist_{\mc K}$ be the induced length metric when all edges have length 1. For real constants $\lambda\geq 1, \epsilon\geq 0$, a combinatorial path $P$ is a \emph{$(\lambda, \epsilon)$-quasigeodesic} if for each subpath $P'$ of $P$ between vertices $u$ and $v$, the length $|P'|$ is at most $\lambda \dist_{\mc K} (u,v) + \epsilon$. A $(\lambda, 0)$-quasigeodesic is called a \emph{$\lambda$-quasigeodesic}.
\end{defn}

\begin{prop}[Simple Ladder]\label{prop:upper}
Let $\mc K$ be a $\delta$-hyperbolic graph. For each $\lambda \geq 1$,  there is an integer $N=N(\delta, \lambda)>0$ such that for all $n> N$  the following property holds:

If $P$ and $Q$ are embedded $\lambda$-quasigeodesics with the same startpoint and endpoint, and with no common interior points.
Then there is an embedded disc diagram $D \rightarrow X_n(\mc K)$ between $P$ and $Q$ such that $D$ is a simple ladder.
\end{prop}

For the proof of Proposition~\ref{prop:upper}, we recall the following well-known fact,  a proof of which can found in
\cite[Chapter III.H, Corollary 1.8]{bridhaef}.
\begin{lem}[Slim rectangles]\label{prop:slim}
Let $\mc K$ be a $\delta$-hyperbolic graph.
For any $\lambda \geq 1$ there exists a constant $\epsilon =\epsilon (\delta, \lambda)>0$ with the following property.
If $P=P_1P_2P_3P_4$ is a closed path such that each $P_i$ is a $\lambda$-quasigeodesic, then each vertex of $P_1$
is contained in the $\epsilon$-neighborhood of the set of vertices of $P_2P_3P_4$.
\end{lem}

\begin{lem}[Fellow traveling]\label{lem:path-splitting}
Let $\mc K$ be a $\delta$-hyperbolic graph.
Let $P$ be an embedded $\lambda$-quasigeodesic, and let $G$ be a geodesic, such that $P,G$ have the same startpoint and endpoint.
Let $\epsilon=\epsilon(\delta,\lambda)$ of Lemma~\ref{prop:slim}.
Suppose  $G$ is the concatenation $G_1 G_2 \cdots  G_{\ell}$ of edge-paths such that:
\[ 10\epsilon \leq |G_i| \leq 20 \epsilon.\]

For each $i$, let $S_i$ be a geodesic from the startpoint of $G_i$ to $P$. Notice that $|S_i| \leq \epsilon$ and is an edge-path.
For $i<\ell$, let $P_i$ be the subpath of $P$ from the endpoint of $S_i$ to the endpoint of $S_{i+1}$, and let
$P_\ell$ be the subpath of $P$ from the endpoint of $S_\ell$ to the endpoint of $P$.

Then $P_i$ and $P_j$ have at most one point in common when $i \neq j$.
Consequently, $P=P_1P_2\dots P_\ell$.
\end{lem}
\begin{proof}
First, $S_i$ is an edge-path for each $i$ since it starts and ends at 0-cells and is embedded. Moreover, $|S_i|<\epsilon$ for each $i$ by Lemma~\ref{prop:slim}, since $PG^{-1}$ is a $\lambda$-quasigeodesic rectangle.
Using this, and applying Lemma~\ref{prop:slim} again, the Hausdorff distance between $G_i$ and $P_i$ is at most $2\epsilon$, since $S_iP_iS_{i+1}^{-1}G_i^{-1}$ is a $\lambda$-quasigeodesic rectangle.

Suppose that $P_i$ and $P_j$ have more than one point in common.  It follows that the minimal distance between $G_i$ and $G_j$ is less than $4\epsilon$.
Since $G$ is a geodesic, if $|i- j|>1$ then the minimal distance between $G_i$ and $G_j$ is at least $10\epsilon$.
Therefore, we can assume that $|i- j|\leq 1$.  Since $P$ is an embedded path, either $P_i$ is contained in $P_j$ or vice-versa. If $P_i \subset P_j$
then applying Lemma~\ref{prop:slim} twice, and using that each $|S_k|<\epsilon$ twice, we see that
$G_i$ is contained in the $4\epsilon$ neighborhood of $G_j$. Since $G$ is a geodesic, this can only happen
if $i=j$.

The second conclusion follows from the first since the concatenation $P_1P_2\cdots P_\ell$ covering $P$ has no backtracks.
\end{proof}

\begin{figure}\centering
\includegraphics[width=.9\textwidth]{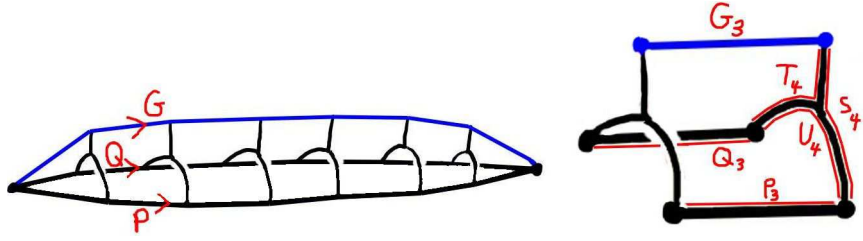}
\caption{Constructing the simple ladder. \label{fig:Dinosaur}}
\end{figure}

\begin{proof}[Proof of Proposition~\ref{prop:upper}]
Suppose that $n> 50\epsilon \lambda$.  Let $G$ be a geodesic between the common startpoints and endpoints of $P,Q$.
We will describe the $1$-skeleton of a disc diagram $D$ between $P$ and $Q$.

If $|G|<10\epsilon$, then $PQ^{-1}$ is a circuit of length at most $20\lambda\epsilon$ and thus bounds a disc diagram with a single $2$-cell yielding
the claim trivially.

We now assume that $|G|\geq 10\epsilon$, and express $G$ as the concatenation $G_1 G_2 \cdots  G_\ell$ with
\[ 10\epsilon \leq |G_i| \leq 20 \epsilon.\]

As described in Lemma~\ref{lem:path-splitting}, the paths $P$  and $Q$ are concatenations  $P_1P_2 \cdots P_\ell$ and $Q_1Q_2\cdots Q_\ell$
with the  following properties (See Figure~\ref{fig:Dinosaur}):
\begin{itemize}
\item For each $i$ there is a geodesic $S_i$ (respectively $T_i$) from the startpoint of $G_i$ to the startpoint of $P_i$ (respectively, the startpoint of $Q_i$)
such that $S_i$ and $P_i$ (respectively $Q_i$) have only one vertex in common.
\item $S_i,T_i$ are edge-paths of length $\leq \epsilon$.
\item Since $S_i$ and $T_i$ are geodesics with the same startpoint,
by possibly re-choose them, we can assume that $S_i=V_i\bar S_i$ and $T_i=V_i\bar T_i$ where $\bar S_i,\bar T_i$ intersect only at their startpoint.
\item The paths $S_i^{-1}T_i \ , \ S_j^{-1}T_j$ are disjoint if $i\not = j$, and hence the same holds for
$\bar S_i^{-1}\bar T_i \ , \ \bar S_j^{-1}\bar T_j$.
\end{itemize}

Let $U_i$ denote the embedded path $\bar S_i^{-1}\bar T_i$ from the startpoint of $P_i$ to the startpoint of $Q_i$. It is immediate that $|U_i| \leq 2\epsilon$. Since $P$ and $Q$ are embedded with disjoint interior, the closed path $C_i = U_iQ_iU_{i+1}^{-1}P_i^{-1}$ is a circuit. Moreover, the circuits $C_i$ and $C_j$ intersect only if $|i-j|\leq 1$.

Notice that
\[ |C_i| \leq 4\epsilon + 2 \lambda (2\epsilon + |G_i|) \leq 48 \epsilon \lambda < n.\]

The union of the circuits $C_1, \dots , C_\ell$ forms the $1$-skeleton of an embedded simple ladder $D$ in $X_n(\mc K)$ between $P$ and $Q$.
\end{proof}

Theorem~\ref{thm:BCP} below is a strong fellow traveling property property for hyperbolic fine graphs admitting a cocompact action. Its proof is an application of the definition of fine graph and  Proposition~\ref{prop:upper}.

\begin{thm}[Strong Fellow Travel Property]\label{thm:BCP}
Let $\mc K$ be a fine hyperbolic graph, and let $\mc G$ be a group acting on $\mc K$ with finitely many orbits of edges.

Suppose the vertex set of $\mc K$ is partitioned into subsets $A$ and $B$ such that no pair of vertices in $A$  are adjacent. 
Let  $\dist$ be a proper metric on $B$ invariant under the action of $G$.

For any $\lambda \geq 1$, there exists a constant $M = M (\lambda) > 0$ with the following property.
If $P_1$ and $P_2$ are embedded $\lambda$-quasigeodesics between the same pair of vertices, then for any $B$-vertex $u$ of $P_1$ there is a $B$-vertex $v$ of $P_2$ such that $\dist (u, v) \leq M$.
\end{thm}
\begin{figure}\centering
\includegraphics[width=.6\textwidth]{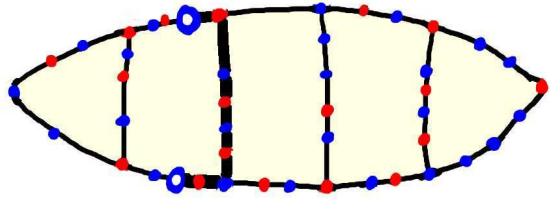}
\caption{The bold subpath of the 2-cell boundary connects a $b$-vertex on one side to a $b$-vertex on the other side. \label{fig:CloseInLadder}}
\end{figure}
\begin{proof}
Let $n$ be sufficiently large so that  $X=X_n(\mc K)$ satisfies the conclusion of Proposition~\ref{prop:upper} for the given $\lambda$.
Since $\mc G$ acts cocompactly on the fine graph $\mc K$, there is a finite number of boundary cycles of 2-cells in $X$ up to the action of $\mc G$.
Combining this with $\dist$ is $\mc G$-equivariant shows that  there is constant $M >0$ such that $\dist (u, v) \leq M$ for any pair of $B$-vertices $u,v$  in the boundary cycle of a 2-cell.

Let $P_1$ and $P_2$ be embedded $\lambda$-quasigeodesics in $\mc K$ with the same startpoint and endpoint.
Without loss of generality, assume that $P_1$ and $P_2$  have no common interior points.

By Proposition~\ref{prop:upper}, there is an embedded simple ladder $D\rightarrow X$ between $P_1$ and $P_2$.
Combining that each 2-cell of $D$ intersects both $P_1$ and $P_2$ in nontrivial arcs, and that no two $A$-vertices are connected by an edge,
it follows that  each  2-cell of $D$ has $B$-vertices in $P_1$ and in $P_2$ (See Figure~\ref{fig:CloseInLadder}).  Since each $B$-vertex of $P_1$ belongs to a 2-cell of $D$, it follows that for any $B$-vertex $u$ of $P_1$ there is a $B$-vertex $v$ of $P_2$ such that $\dist (u, v) \leq M$.
\end{proof}

\section{Equivalences of Formulations of Relative  quasiconvexity}\label{sec:quasiconvexity}

In this section we prove Theorem~\ref{thm:main} on the equivalence between relative quasiconvexity of Definition~\ref{def:ours}, labelled by $(\text{Q-0})$, and relative quasiconvexity of Definition~\ref{def:O}, labelled by $(\overline{\text{Q-1}})$,  in the context of countable relatively hyperbolic groups.

The argument introduces  two auxiliary definitions of relative quasiconvexity, labelled by $\widehat{\text{Q-1}}$ and $\widehat{\text{Q-0}}$, and then the theorem follows after proving the following equivalences:
\begin{equation}\label{eq:equivalences}
 \overline{\text{Q-1}}  \Longleftrightarrow   \widehat{\text{Q-1}} \Longleftrightarrow  \widehat{\text{Q-0}}  \Longleftrightarrow   \text{Q-0} .
\end{equation}

The section is divided in six short parts as follows.  First  we recall the notion of coned-off Cayley graph, and deduce a strong version of the fellow travel property using  the main result  of Section~\ref{sec:ladder}. The second part shows that relative quasiconvexity of Definition~\ref{def:O}, labelled by $(\overline{\text{Q-1}})$, implies finite relative generation. The third part state the two auxiliary definitions.  The remaining three parts correspond to each of the three equivalences in illustration~\eqref{eq:equivalences}.

For the rest of the section, let $\mc G$ be a hyperbolic group relative to a collection of subgroups  $\mb P$, let $S$ be a finite relative generating set for $(\mc G, \mb P)$, and let $\dist$ be a proper left-invariant metric on $\mc G$.

\subsection{The Coned-off Cayley Graph}

\begin{defn}[Cayley Graph]
Let $S$ be a subset of a group $\mc G$, and assume that $S$ is closed under inverses. The \emph{Cayley graph $\Gamma (\mc G, S)$} is an oriented labelled 1-complex with vertex set $G$ and edge set $\mc G\times S$. An edge $(g,s)$ goes from the vertex $g$ to the vertex $gs$ and has label $s$. Observe that $\Gamma (\mc G, S)$ is connected if and only if $S$ is a generating set of $\mc G$.
\end{defn}

The notion of coned-off Cayley graph is originally due to Farb~\cite{Fa98} and the following generalization is taking from~\cite[Sec. 3.4]{HK08}.

\begin{defn}[Coned-off Cayley Graph $\widehat \Gamma$]
Let $S$ be a finite relative generating set for $(\mc G, \mb P)$ and assume that $S$ is closed under inverses. The \emph{coned-off Cayley graph} $\gps$ is the graph constructed from the Cayley graph $\Gamma=\Gamma (\mc G, S)$ as follows: For each left coset $g\mc P$ with $g \in \mc G$ and $\mc P \in \mb P$, add a new vertex $v(g\mc P)$ to $\Gamma$, and add a $1$-cell from this new vertex to each element of $g \mc P$. These new vertices of $\gps$ that are not in $\Gamma$ are called \emph{cone-vertices}. Each $1$-cell  of $\gps$ between an element of $\mc G$ and a cone vertex is a \emph{cone-edge}. Note that the coned-off Cayley graph $\gps$ is connected since $S$ is a relative generating set for $(\mc G,\mb P)$.

There is a related (genuine) Cayley graph of $\mc G$ with respect to the generating set defined as the disjoint union $S \sqcup \bigsqcup_i \mc P_i$ which we denote by $\overline \Gamma (\mc G, S\cup \mb P)$.
\end{defn}

Proposition~\ref{prop:Dahmani} below appeared in Dahmani's thesis~\cite[Proof Lemma A.4]{Da03} and in Hruska's work~\cite[Proof (RH-4) $\Rightarrow$ (RH-5)]{HK08}. We included a proof using the results of Section~\ref{sec:fine-graphs}.
\begin{prop}\label{prop:Dahmani}
Suppose that $\mc K$ is a $(\mc G, \mb P)$-graph. If $S$ is a finite relative generating set for $(\mc G, \mb P)$, then there exists a $(\mc G, \mb P)$-graph $\mc K'$
such that $\mc K$ and $\gps$ both embed equivariantly and simplicially into $\mc K'$.

In particular, $\gps$ is a $(\mc G, \mb P)$-graph.
\end{prop}
\begin{figure}\centering
\includegraphics[width=.25\textwidth]{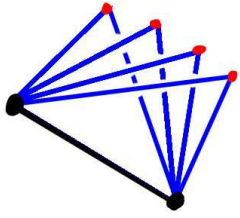}
\caption{Adding triangles around each edge. \label{fig:TrianglesAround}}
\end{figure}
\begin{proof}
We will perform finitely many arc $\mc G$-attachments to $\mc K$ to obtain a graph $\mc K'$ where $\widehat \Gamma= \gps$ equivariantly embeds in $\mc K'$. Then Lemmas~\ref{prop:arc-attachment-2} and~\ref{prop:arc-attachment} will imply that $\mc K'$ is a $(\mc G, \mb P)$-graph, and that $\mc K \hookrightarrow  \mc K'$ is a quasi-isometric embedding.

The graph $\mc K'$ is obtained as follows. By $\mc G$-attaching a triangle around an edge of $\mc K$, we obtain a new $\mc K_1$ where $\mc G$ acts freely on triangle-tops. See Figure~\ref{fig:TrianglesAround}.  We identify $\mc G$ with a triangle-top orbit.  Performing an additional edge $\mc G$-attachment to $\mc K_1$ for each element of $S$ yields a graph $\mc K_2$.  Finally, for each group $\mc P \in \mb P$, choose a vertex $v_{_{\mc P}} \in \mc K$ stabilize by $\mc P$ and perform a $\mc G$-attachment of an edge between $1_{\mc G}$ and $v_{_{\mc P}}$ of $\mc K_2$ to obtain a graph $\mc K'$.  Observe that $\widehat \Gamma$ equivariantly embeds in $\mc K'$, and that $\mc K'$ is a $(\mc G, \mb P)$-graph by Lemmas~\ref{prop:arc-attachment-2} and~\ref{prop:arc-attachment}.

By Lemma~\ref{prop:bijection}, all infinite valence vertices of $\mc K'$ are contained in $\widehat \Gamma$.
Since $\mc K'$ has finitely many $\mc G$-orbits of edges, $\widehat \Gamma$ is obtained from $\mc K'$ after finitely many $\mc G$-removals of vertices of finite valence together with their adjacent vertices. Since $\mc K'$ is a $(\mc G, \mb P)$-graph,  Lemma~\ref{prop:removals} implies that $\gps$ is a $(\mc G, \mb P)$-graph and that $\widehat \Gamma \hookrightarrow  \mc K'$ is a $\mc G$-equivariant quasi-isometry.
\end{proof}

\begin{prop}[Strong Fellow Travel for $\widehat \Gamma$]\label{prop:BCP}
Let $S$ be a finite relative generating set for $(\mc G, \mb P)$, and let $\dist$ be a proper left-invariant metric on $\mc G$.

For any $\lambda \geq 1$, there exists a constant $M = M (\lambda) > 0$ with the following property.
If $P_1$ and $P_2$ are embedded $\lambda$-quasigeodesics in $\widehat \Gamma$ between the same pair of elements of $\mc G$,
then for each element $g _1\in \mc G$ of $P_1$, there is an element $g_2\in \mc G$ of $P_2$ such that $\dist (g_1, g_2) \leq M$.
\end{prop}
\begin{proof} \label{rem:BCPforHat}
By Proposition~\ref{prop:Dahmani}, the coned-off Cayley graph $\gps$ is a $(\mc G, \mb P)$-graph.
Proposition~\ref{prop:BCP} follows from Theorem~\ref{thm:BCP}  by declaring $B =\mc G$  and  $A$ to be the collection of cone-vertices.
\end{proof}


Corollary~\ref{prop:BCP-2} below is an analogous result to Proposition~\ref{prop:BCP} for the Cayley graph $\bar \Gamma (\mc G, S \cup \mb P)$.
A version of this result for the case that $\mc G$ is finitely generated appears as~\cite[Prop. 3.15]{Os06}. Its hypothesis requires that the quasigeodesics do not backtrack in the following sense based on~\cite[Def. 3.9]{Os06}:

\begin{defn}[Phase Vertices and Backtracking in $\bar \Gamma$] \label{defn:backtracking}
Let $Q$ be a path in $\bar \Gamma (\mc G, S \cup \mb P)$.  For a subgroup $\mc P \in \mb P$, a nontrivial subpath $U$ of $Q$ is called a \emph{$\mc P$-component of $Q$} if all edges of $U$ are labelled by elements of $\mc P$, in particular, all vertices of $U$ belong to the same left coset of $\mc P$.

An $\mc P$-component of $Q$ is called \emph{maximal} if it is not a proper subpath of $\mc P$-component of $Q$. The path $Q$ is said to \emph{backtrack} if $Q$ has two disjoint and maximal $\mc P$-components for some $\mc P \in \mb P$.  A vertex $v$ of $Q$ is called a \emph{phase vertex} if $v$ is not an interior vertex of a $\mc P$-component of $Q$.
\end{defn}
\begin{rem}
Since $\bar \Gamma$ is a Cayley graph of $G$ with respect to the disjoint union $S\sqcup \bigsqcup_{\mc P \in \mb P} \mc P$,
a $\mc P$-component and a $\mc P'$-component of $Q$ can not have edges in common if $\mc P$ and $\mc P'$ are different.
\end{rem}

\begin{lem}[$\bar \Gamma$ and $\widehat \Gamma$ are quasi-isometric]\label{lem:qi}
The identity map of $\mc G$ induces a natural $(2,0)$-quasi-isometry $\varphi: \bar \Gamma \rightarrow \widehat \Gamma$ with the following property.
If $P$ is an embedded path, with only phase vertices, and without backtracking in $\bar \Gamma$, then $\varphi (P)$ is embedded in $\widehat \Gamma$.
\end{lem}
\begin{proof}
The quasi-isometry $\varphi$ maps edges labelled by elements of $S$ to the corresponding edge in $\widehat \Gamma$,
and edges labelled by an element of $\mc P \in \mb P$ are map to a $2$-path passing through a cone-vertex of a left coset of $\mc P$.
Observe that if $P$ is embedded and $\varphi (P)$ is not embedded, then either $P$ contains a $\mc P$-component of length at least two, or $P$ backtracks.
\end{proof}

\begin{cor}[Strong Fellow Travel for $\bar \Gamma$]\label{prop:BCP-2}
Let $S$ be a finite relative generating set for $(\mc G, \mb P)$, and let $\dist$ be a proper left-invariant metric on $\mc G$.

For any $\lambda \geq 1$, there exists a constant $M = M (\lambda) > 0$ with the following property.
If $P_1$ and $P_2$ are embedded $\lambda$-quasigeodesics without backtracking in $\bar \Gamma$ between the same pair of elements of $\mc G$,
then for each phase vertex $g_1$ of $P_1$, there is a phase vertex $g_2$ of $P_2$ such that $\dist (g_1, g_2) \leq M$.
\end{cor}
\begin{proof}
Let $\varphi : \bar \Gamma \rightarrow \widehat \Gamma$ be the quasi-isometry given by Lemma~\ref{lem:qi}.
If $P_1$ and $P_2$ contain only phase vertices, the conclusion follows immediately after mapping $P_1$ and $P_2$ to $\widehat \Gamma$ and applying Proposition~\ref{prop:BCP}.

The Corollary holds for general $P_1$ and $P_2$ after the following observation. Let $P$ be a $\lambda$-quasigeodesic in $\bar \Gamma$, and let $P'$ be the path obtained from $P$ after replacing by a single edge each maximal  $\mc P$-component (for each $\mc P \in \mb P$). The new path $P'$ contains only phase vertices, its vertex set equals the set of phase vertices of $P$, and $P'$ is also a $\lambda$-quasigeodesic (since the process from $P$ to $P'$ is only shortening distances).
\end{proof}

\subsection{Finite Relative Generation} 

\begin{lem}[Bounded Intersection]\cite{HK08, MP09} \label{lem:quasiorthogonality1}
Let $A$ be a countable group with a proper left-invariant metric $\dist$.
Then for each $g \in C$, for each pair of subgroups $B$ and $C$ of $A$, and for each constant $K \geq 0$,
there exists a constant $M = M(B,C, g, \dist,  K) \geq 0$ so that
\[  B \cap N_{K}(gC)  \subset  N_M ( B \cap gCg^{-1}), \]
where $N_{K}(gC)$ and $N_M (B \cap gCg^{-1})$ denote the closed $K$-neighborhood and the closed $M$-neighborhood
of $gC$ and $B\cap gCg^{-1}$ in $(A,\dist)$.
\end{lem}
\begin{proof}
Suppose the statement is false for the constant $K$. Then there are sequences $\{ q_n \}_{n=1}^\infty$
and $\{ h_n \}_{n=1}^\infty$ such that $q_n \in B$, $q_n h_n \in gC$, $d(1,h_n) \leq K$, and
\[ d(q_n, B\cap gCg^{-1} ) \geq n .\]
Since balls are finite in the metric space $(A,d)$, without loss of generality
assume $\{ h_n \}_{n=1}^\infty$ is a constant sequence $\{ h\}_{n=1}^\infty$.
For any $m $ and $n$, observe that $q_nq_m^{-1}
= (q_nh) (q_m h)^{-1} \in B\cap gCg^{-1}$, and hence $q_mh$ and $q_nh$
are in the same right coset of $B\cap gCg^{-1}$, say $ (B\cap gCg^{-1}) f$. It
follows that
\[ d(q_n, B \cap gCg^{-1} ) \leq d(q_n, q_nh) + d(q_nh, B \cap gCg^{-1}) \leq K + d(1,f) \]
for any $n$, a contradiction.
\end{proof}

\begin{lem}[Parabolic Approximation]\label{lem:parabolic-aprox}
For each subgroup $\mc H < \mc G$ and each $\sigma \geq 0$, there is  $L=L(S, \mc H, \dist, \sigma)>0$ with the following property:
If $h \in \mc H$ is a product of the form $h =g p f$ where $\dist  (1, g) \leq \sigma$, $\dist  (1, f) \leq \sigma$, and $p \in \mc P$ for some $\mc P \in \mb P$. Then $h=a b$ where $a \in \mc H \cap gPg^{-1}$,  $ b \in \mc H$, and $\dist (1, b) \leq L$. (See Figure~\ref{fig:ConjugateComputation}.)
\end{lem}

\begin{figure}\centering
\includegraphics[width=.4\textwidth]{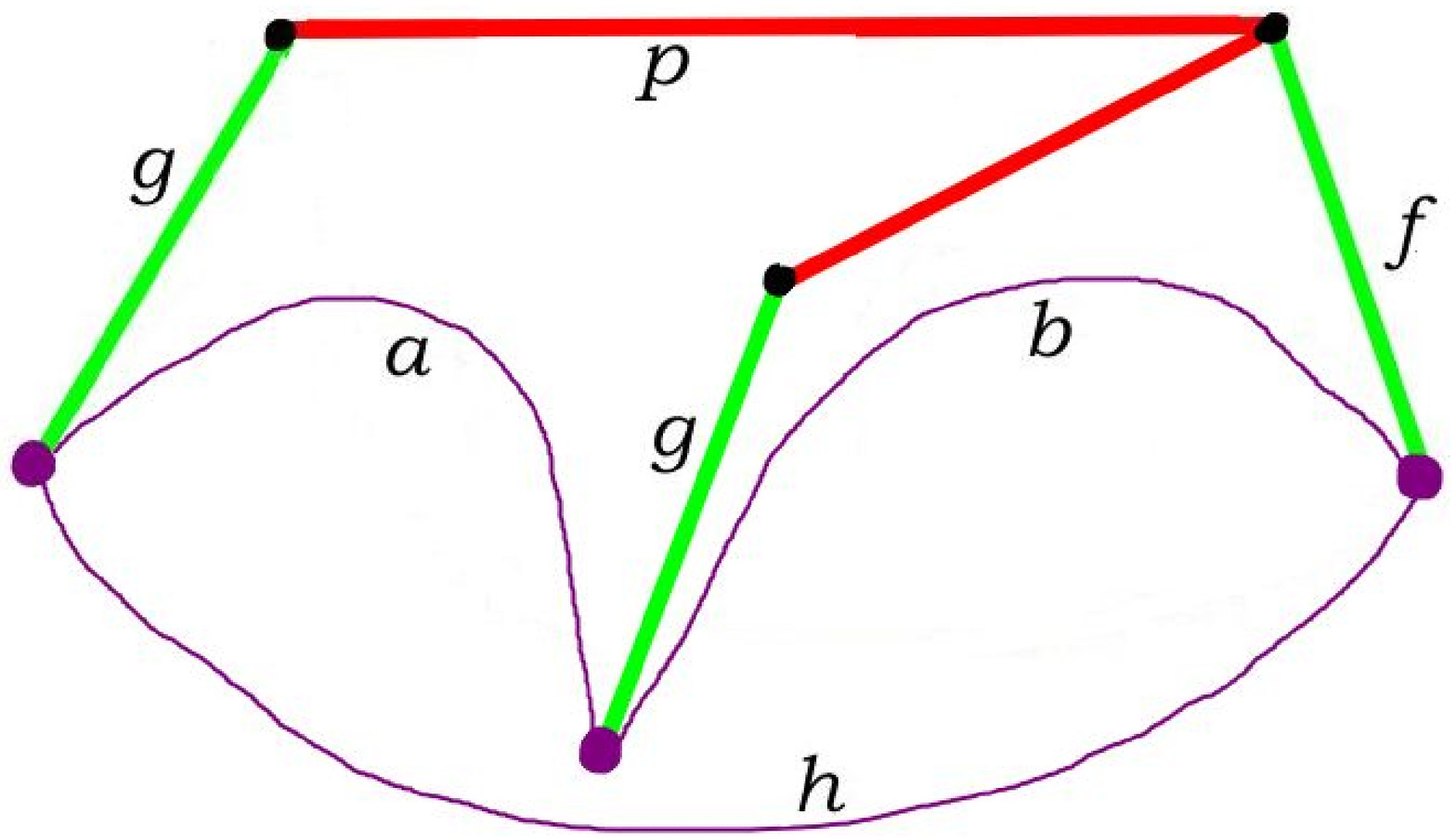}
\caption{$h=ab$ of Lemma~\ref{lem:parabolic-aprox} \label{fig:ConjugateComputation}}
\end{figure}

\begin{proof}
For each $g \in G$ and $\mc P\in \mb P$, let $M_{\mc P, g} =M(\mc H, \mc P,  g, \dist, \sigma)$ be the constant provided by Lemma~\ref{lem:quasiorthogonality1}. Since $\dist$ is proper and $\mb P$ is a finite collection of subgroups,
\[  L = \max \{   M_{\mc P, g}\ :\  g \in \mc G,\ \dist (1, g) \leq \sigma,\ \mc P\in \mb P \} \]
is a well-defined positive integer.

Suppose that $h \in \mc H$ is a product of the form $h =g p f$ where $\dist (1, g) \leq \sigma$, $\dist (1, f) \leq \sigma$, and $p \in \mc P$ for $\mc P \in \mb P$. Observe that
\[ h \in  \mc H \cap N_{\sigma}(gP) \subset N_{L}( \mc H \cap gPg^{-1} ), \]
where the neighborhoods are in $\mc G$ with respect to the metric $\dist$, and the second inclusion is a consequence of Lemma~\ref{lem:quasiorthogonality1}.
The conclusion of the lemma is immediate.
\end{proof}

\begin{prop}[$(\overline{\text{Q-1}}) \Rightarrow$ Finite Relative Generation] \label{prop:generation}
Suppose that $\mc H< \mc G$ satisfies Definition~\ref{def:O} of relative quasiconvexity for a proper left-invariant metric $\dist$ on $\mc G$, and  a constant $\sigma$.
Then $\mc H$ is finitely generated relative to the finite collection of subgroups
\[ \mb R = \{ \mc H \cap g \mc P g^{-1} \ :\ g\in \mc G,\ \dist (1, g)\leq \sigma,\ \mc P \in \mb P  \} .\]
\end{prop}
\begin{proof}
Let $ K = \max \{  \dist  (1, g)  : g \in S \}$, and
let $L=L(S, \mc H,  \dist, \sigma)$ be the constant provided by Lemma~\ref{lem:parabolic-aprox}.
Let \[ T= \big \{ h \in \mc H :  \dist (1, h) \leq  \max \{ 2\sigma + K, L\}  \big \} .\]
We will show that $T$ is a finite relative generating set for $\mc H$ with respect to the finite collection of subgroups $\mb R$.

Let $g \in \mc H$, and let $Q=Q_1 \cdots Q_n$ be a geodesic in $\overline \Gamma$ from the identity element to $g$. Here each $Q_i$ is an (oriented) edge of $Q$.
For each $i$, there is an (oriented) path $X_i$ from the startpoint of $Q_i$ to an element of $\mc H$ such that the $\dist$-distance between its endpoints is at most $\sigma$. Let $g_i$ be the element of $\mc G$ defined as the difference between the endpoint and startpoint of $X_i$, and $q_i \in S\cup \bigcup_{\mc P \in \mb P} \mc P$ the label of the (oriented) edge $Q_i$. Then let $h_i = g_i^{-1}q_i g_{i+1}$, and  notice that $g =h_1 \cdots h_n$ and $\dist (1, g_i) \leq \sigma$ for each $i$. See Figure~\ref{fig:CloseToH}.

If $Q_i$ is labelled by an element of $S$, that is $q_i \in S$, then
\[ \dist (1, h_i) \leq  \dist (1, g_i) + \dist (1, q_i) + \dist (1, g_{i+1}) \leq 2\sigma + K,\]
and hence $h_i \in T$.

Suppose that $Q_i$ is labelled by a parabolic element, that is $q_i \in \bigcup_{\mc P \in \mb P} \mc P$. Then the element $h_i$ can be approximated by a parabolic. Namely, by Lemma~\ref{lem:parabolic-aprox}, $h_i=a_i b_i$ where $a_i \in \mc H \cap gPg^{-1}$,  $ b_i \in \mc H$, and $\dist (1, b_i) \leq L$. Hence $h_i$ is a product of an element of a subgroup in $\mb R$, and an element  of $T$.

It follows that each element $g \in \mc H$ can be expressed as a product of elements of $T$ and elements of the subgroups in $\mb R$.
\end{proof}

\begin{figure}\centering
\includegraphics[width=.6\textwidth]{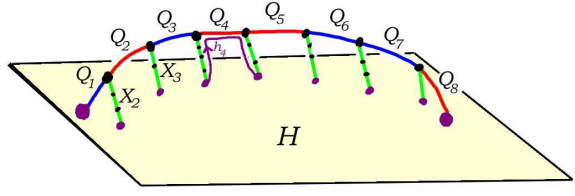}
\caption{A geodesic $Q$ with endpoints in quasiconvex subgroup.\label{fig:CloseToH}}
\end{figure}

\subsection{Auxiliary Definitions}

Since $\mc G$ acts on the connected graph $\mc K$ with finitely many orbits of edges, an standard argument shows that $\mc G$ is finitely generated relative to $\mb P$. This justifies the existence of finite relative generating sets in the following two definitions.

\begin{defn}[($\widehat{\text{Q-1}}$) Auxiliary definition in $\widehat \Gamma$]\label{def:hat}
Let $S$ be a finite relative generating set for $(\mc G, \mb P)$, and let $\dist$ be a proper left-invariant metric on $\mc G$.  A subgroup $\mc H$ of $\mc G$ is \emph{quasiconvex relative to  $\mb P$} if there exists a constant $\sigma \geq 0$ such that the following holds: Let $f$, $g$ be two elements of $\mc H$, and let $P$ be an arbitrary geodesic path from $f$ to $g$ in $\gps$. For any non-cone vertex $p$ of $P$, there exists a non-cone vertex $h$ in $\mc H$ such that $\dist (p,h) \leq \sigma$.
\end{defn}

\begin{defn}[$(\widehat{\text{Q-0}})$ Second auxiliary definition in $\widehat \Gamma$]\label{def:Revisited-hat}
Let $S$ be a finite relative generating set for $(\mc G, \mb P)$.  A subgroup $\mc H$ of $\mc G$ is \emph{quasiconvex relative to  $\mb P$} if  there is a non-empty connected $\mc H$-invariant subgraph $\mc L$ of $\widehat \Gamma$ that is  quasi-isometrically embedded and has finitely many $\mc H$-orbits of edges.
\end{defn}

\subsection{The equivalence $ (\overline{ \text{Q-1}})  \Leftrightarrow  (\widehat{\text{Q-1}})$}

\begin{prop}[$ (\overline{ \text{Q-1}})  \Leftrightarrow  (\widehat{\text{Q-1}})$]  \label{prop:hatEosin}
Relative quasiconvexity of Definition~\ref{def:O} and relative quasiconvexity of Definition~\ref{def:hat} are equivalent notions.

In particular, Definition~\ref{def:hat} is independent of the generating set $S$ and the left-invariant metric $\dist$ on $G$.
\end{prop}

\begin{proof}
Let $\varphi: \bar \Gamma \rightarrow \widehat \Gamma$ be the natural quasi-isometry induced by the identity map on $G$, see Lemma~\ref{lem:qi}.

Suppose that $\mc H < \mc G$ satisfies Definition~\ref{def:hat} for a constant $\sigma$. Let $\bar P$ be a geodesic in $\bar \Gamma$ between elements of $\mc H$. Then $\varphi$ maps $\bar P$ to a $2$-quasigeodesic $\widehat P$. Let $\widehat Q$ be a geodesic in $\widehat \Gamma$ between the endpoints of $\widehat P$.
By Proposition~\ref{prop:BCP}, there is a constant $M= M(2)$  such that for each non-cone vertex $p$ of $\widehat P$, there is a non-cone vertex $q \in \widehat Q$ such that $\dist (p, q) \leq M$. By Definition~\ref{def:hat}, for each non-cone vertex $q$ of $\widehat Q$, there is $h \in \mc H$ such that $\dist (q,h) \leq \sigma$. Since the vertices of $\bar P$ map to the non-cone vertices of $\widehat P$, it follows that for each vertex $p$ of $\bar P$, there is a vertex $h$ of $\mc H$ such that $\dist (p, h) \leq \sigma+M$. Therefore, $\mc H$ satisfies Definition~\ref{def:O}.

Conversely, suppose that $\mc H < \mc G$ satisfies Definition~\ref{def:O} for a constant $\sigma$, and let $\hat P$ be a geodesic in $\widehat \Gamma$ between elements of $\mc H$.  Let $\bar Q$ be a geodesic in $\bar \Gamma$ between the endpoints of $\hat P$. Observe that there is an embedded $2$-quasigeodesic $\bar P$ in $\bar \Gamma$ that maps onto $\hat P$ by $\varphi$, has only phase vertices, and does not backtrack. Then Proposition~\ref{prop:BCP-2} implies that there is a constant $M=M(2)$ such that the vertices of $\bar Q$ and $\bar P$ (and hence $\hat P$) are $M$-closed with respect to the metric $\dist$. Since the vertices of $\bar Q$ are $\sigma$-close to the elements of $\mc H$ with respect to $\dist$, we conclude that $\mc H$
satisfies Definition~\ref{def:hat} for the constant $\sigma+M$.
\end{proof}

\subsection{The equivalence $(\widehat{\text{Q-1}})    \Leftrightarrow   (\widehat{\text{Q-0}})$}

\begin{prop} [$ (\widehat{\text{Q-1}})  \Rightarrow  (\widehat{\text{Q-0}})$]
Relative quasiconvexity of  Definition~\ref{def:hat} implies relative quasiconvexity of Definition~\ref{def:Revisited-hat}.
\end{prop}
\begin{proof}
Let $\dist$ be a proper left-invariant metric  on $G$, and let $S$ be a finite relative generating set of $(\mc G, \mb P)$.
Suppose that $\mc H < \mc G$ satisfies  Definition~\ref{def:hat} for some $\sigma>0$.

\emph{Choosing the  subgraph $\mc L$.}
By Proposition~\ref{prop:hatEosin}, $\mc H$ also satisfies relative quasiconvexity of Definition~\ref{def:O}. Thus by Proposition~\ref{prop:generation}, $\mc H$ is generated by a finite subset $T\subset \mc H$ relative to the finite collection of subgroups
\[ \mb R = \{ \mc H \cap g \mc P g^{-1} \ :\ g\in \mc G,\ \dist (1, g) \leq \tau ,\ \mc P \in \mb P  \} ,\]
where $\tau$ is positive constant. Without loss of generality assume that $\tau \geq \sigma$.

Let $C$ be the cone-vertices associated to the subgroups in $\mb R$. Let $\mc J$ be a finite connected subgraph of $\widehat \Gamma$ containing $\{ 1_{\mc G}\} \cup T \cup C$. By  Lemma~\ref{lem:initial-complex} the graph $\mc L = \bigcup_{h\in \mc H} h \mc J$ is connected.

\emph{The inclusion $\mc L \subset \widehat \Gamma$ is a quasi-isometric embedding with respect to the combinatorial path metrics.} Let $\dist_{\widehat \Gamma }$ and $\dist_{\mc L}$ denote the combinatorial path metrics of $\widehat \Gamma$ and $\mc L$ respectively.  Specifically we will show that there is a  constant  $N>0$ such that for any $h \in \mc H$ we have $ \dist_{\mc L} (1, h) \leq N \dist_{\widehat \Gamma}(1, h) $.

We define our constant $N>0$ with the help of four auxiliary constants.  Let
$ J = \max \{ \dist_{\mc L} (1, c) \ :\  c \in C \}$,
and observe that $J<\infty$  since $|C| < \infty$, $\mc L$ is connected, and $1_{\mc G} \in \mc L$.
 Let
$ K = \max \{  \dist  (1, g)  \ :\  g \in S \}$ ,
and notice that is a finite number since $|S|<\infty$.
 Let  $ L=L(S, \mc H, \dist, \sigma)>0 $
be the constant provided by Lemma~\ref{lem:parabolic-aprox}.
 Let
\[ M = \max \left \{\dist_{\mc L } (1, h)  \ :\  h \in \mc H,\  \dist (1, h) \leq  2\sigma + K +  L  \right \},\]
and notice that is finite since $\dist$ is a proper metric.
Let $ N = \max \{ 2J+L, M  \}$.

Let $h \in \mc H$ and let $P$ be a geodesic in $\widehat \Gamma (\mc G, \mb P, S)$ from $1_{\mc G}$ to $h$. Express $P = P_1P_2\cdots P_\ell$ as a concatenation of paths so that each $P_i$ is either a single edge with endpoints in $\mc G$, or is a \emph{shortcut} consisting of two cone-edges meeting at a cone-point. Note that each shortcut has also endpoints in $\mc G$). For each $i$ let $p_i$ denote the endpoint of $P_i$.

Since $\mc H$ satisfies Definition~\ref{def:hat} for the constant $\sigma$, for each $0 < i < \ell$,  there is an element $h_i \in \mc H$ with $\dist (p_i, h_i) \leq \sigma$. Let $h_0=1_{\mc G}$ denote the startpoint of $P$ and let $h_\ell=h$ denote the endpoint of $P$.

If $P_i$ is a single edge, then $\dist (h_{i-1}, h_i) \leq 2 \sigma + K$, and therefore \[\dist_{\mc L } (h_{i-1}, h_i) \leq M \leq N.\]

If $P_i$ is a shortcut, then $h_{i-1}^{-1}h_i = g p f$ where $\dist  (1, g) \leq \sigma$, $\dist  (1, f) \leq \sigma$, and $p \in \mc P$ for some $\mc P \in \mb P$.
By Lemma~\ref{lem:parabolic-aprox}, $h_{i-1}^{-1}h_i=a b$ where $a \in \mc H \cap gPg^{-1}$,  $ b \in \mc H$, and $\dist (1, b) \leq L$.
Since $\dist (1, g) \leq \sigma \leq \tau$, the cone-vertex $c_{gP}$ is contained in $\mc L$, and therefore
$ \dist_{\mc L}(1, a) \leq 2\dist_{\mc L} (1, c_{g P}) \leq  2J. $
Hence
\[ \dist_{\mc L}(h_{i-1},h_i) \leq \dist_{\mc L}(1, a) + \dist_{\mc L}(1, b) \leq 2J+L \leq N.\]

To conclude, observe that
\[  \dist_{\mc L}(1, h) \leq \sum_{i=1}^\ell \dist_{\mc L} (h_{i-1}, h_i) \leq  N \ell \leq N \dist_{\widehat \Gamma} (1, h). \qedhere \]
\end{proof}

\begin{prop} [$ (\widehat{\text{Q-0}})  \Rightarrow  (\widehat{\text{Q-1}})$]\label{prop:q2hat-implies-q1hat}
Relative quasiconvexity of Definition~\ref{def:Revisited-hat} implies relative quasiconvexity of Definition~\ref{def:hat}.
\end{prop}
\begin{proof}
Let $\dist$ be a proper left-invariant metric  on $G$, and let $S$ be a finite relative generating set of $(\mc G, \mb P)$.
Suppose that $\mc H < \mc G$ satisfies Definition~\ref{def:Revisited-hat}.

Let $\mc L$ be connected quasi-isometrically embedded subgraph $\mc L$ of $\widehat \Gamma = \gps$ that is $\mc H$-equivariant and has  finitely many $\mc H$-orbits of edges.  The following statements about $\mc L$ hold.
\begin{enumerate}
\item Since $\mc L$ is connected, it has at least one vertex that is not a {cone-vertex}. Without loss of generality, assume that the identity $1_{\mc G} \in \mc G$ is a vertex of $\mc L$. In particular, we assume that all elements of $\mc H$ are vertices of $\mc L$.
\item \label{stat-2} Since $\mc H$ acts cocompactly on $\mc L$, there is $\kappa>0$ such that each vertex of $\mc L$ is either a cone-vertex, or an element of $\mc G$ is at distance at most $\kappa$ from $\mc H$ with respect to $\dist$.
\item \label{stat-3} Since $\mc L$ is quasi-isometrically embedded subgraph of $\widehat \Gamma$, there is $\lambda \geq 1$ such that any geodesic in $\mc L$ is a $\lambda$-quasi-geodesic in $\widehat \Gamma$.
\end{enumerate}

Let $M = M (\lambda) > 0$ be the constant provided by Proposition~\ref{prop:BCP}. Let $Q$ be a geodesic in $\widehat \Gamma$ with endpoints in $\mc H$.
Let $P$ be a geodesic in $\mc L$ between the endpoints of $Q$. By Statement~\eqref{stat-3},  $P$ is a $\lambda$-quasigeodesic. Therefore,  Proposition~\ref{prop:BCP} implies that for each non-cone vertex $q$ of $Q$, there is a non-cone vertex $p$ of $P$ such that $\dist (p, q) \leq M$.  Combining this with  Statement (\ref{stat-2}) shows that for each non-cone vertex $q$ of $Q$ there is a  non-cone vertex $h \in \mc H$ such that $\dist (p, h) \leq M+\kappa$. Therefore, $\mc H$ satisfies Definition~\ref{def:hat}.
\end{proof}

\subsection{The equivalence $ (\widehat{\text{Q-0}})  \Leftrightarrow  ({\text{Q-0}})   $ }

\begin{prop}[$ (\widehat{\text{Q-0}})  \Leftrightarrow  ({\text{Q-0}})   $] \label{prop:oursEhat}
Relative quasiconvexity of Definition~\ref{def:Revisited-hat} is equivalent to relative quasiconvexity of Definition~\ref{def:ours}.
\end{prop}
\begin{proof}
By Proposition~\ref{prop:Dahmani}, the coned-off Cayley graph $\gps$ with respect to a finite relative generating set $S$ is a $(\mc G, \mb P)$-graph.
By Theorem~\ref{thm:movings}, a subgroup $\mc H$ satisfies relative quasiconvexity of Definition~\ref{def:ours} for a $(\mc G, \mb P)$-graph $\mc K$ if and only if it does for $\gps$.
\end{proof}

%
%
%
\bibliographystyle{plain}

\bibliography{xbib}

\begin{thebibliography}{10}

\bibitem{BO99}
B.H. Bowditch.
\newblock Relatively hyperbolic groups.
\newblock preprint, 1999.

\bibitem{bridhaef}
Martin~R. Bridson and Andr{\'e} Haefliger.
\newblock {\em Metric spaces of non-positive curvature}, volume 319 of {\em
  Grundlehren der Mathematischen Wissenschaften [Fundamental Principles of
  Mathematical Sciences]}.
\newblock Springer-Verlag, Berlin, 1999.

\bibitem{Da03}
F.~Dahmani.
\newblock Les groupes relativement hyperboliques et leurs bords.
\newblock PhD thesis, Univ. Louis Pasteur, Strasbourg, France, 2003.

\bibitem{Fa98}
B.~Farb.
\newblock Relatively hyperbolic groups.
\newblock {\em GAFA, Geom. funct. anal.}, 8(5):810--840, 1998.

\bibitem{HK08}
G.~Christopher Hruska.
\newblock Relative hyperbolicity and relative quasiconvexity for countable
  groups.
\newblock {\em Algebr. Geom. Topol.}, 10(3):1807--1856, 2010.

\bibitem{MM09}
Jason~Fox Manning and Eduardo Mart{\'{\i}}nez-Pedroza.
\newblock Separation of relatively quasiconvex subgroups.
\newblock {\em Pacific J. Math.}, 244(2):309--334, 2010.

\bibitem{MP09}
Eduardo Mart{\'{\i}}nez-Pedroza.
\newblock Combination of quasiconvex subgroups of relatively hyperbolic groups.
\newblock {\em Groups Geom. Dyn.}, 3(2):317--342, 2009.

\bibitem{MaWi10}
Eduardo Mart{\'{\i}}nez-Pedroza and Daniel~T. Wise.
\newblock Local quasiconvexity of groups acting on small cancellation
  complexes.
\newblock {\em J. Pure Appl. Algebra.}, to appear.
\newblock Preprint at \url{arXiv:1009.3407}.

\bibitem{Os06}
Denis~V. Osin.
\newblock Relatively hyperbolic groups: intrinsic geometry, algebraic
  properties, and algorithmic problems.
\newblock {\em Mem. Amer. Math. Soc.}, 179(843):vi+100, 2006.

\bibitem{Tu98}
Pekka Tukia.
\newblock Conical limit points and uniform convergence groups.
\newblock {\em J. Reine Angew. Math.}, 501:71--98, 1998.

\end{thebibliography}


\end{document}